\documentclass[12pt,reqno]{amsart}
\usepackage{epsfig}
\usepackage{amsmath}
\usepackage{amssymb, amsthm}
\usepackage{graphicx}
\usepackage{color}
\usepackage{psfrag}

\newcommand{\Z}{\mathbb{Z}}
\newcommand{\R}{\mathbb{R}}

\numberwithin{equation}{section}

\newtheorem{maintheorem}{Theorem}

\newtheorem{proposition}[subsection]{Proposition}
\newtheorem{theorem}[subsection]{Theorem}
\newtheorem{lemma}[subsection]{Lemma}

\begin{document}
%\thanks{ This work was mostly done at the Instituto de Ci\^encias de Matem\'atica e de Computa\c{c}\~ao (ICMC)
%and finished when A. Tahzibi was on leave from ICMC-USP for a post-doctoral at
%IMERL-Uruguay supported by CNPq-Brazil. M. Bronzi thanks the support of  Fapesp-Brazil with a doctoral fellowship.}
\author{M. Bronzi}
\address{Departamento de Matem\'atica,
  ICMC-USP S\~{a}o Carlos, Caixa Postal 668, 13560-970 S\~{a}o
  Carlos-SP, Brazil.}
\email{bronzi@icmc.sc.usp.br}

\author{A. Tahzibi}
\address{Departamento de Matem\'atica,
  ICMC-USP S\~{a}o Carlos, Caixa Postal 668, 13560-970 S\~{a}o
  Carlos-SP, Brazil.}
\email{tahzibi@icmc.sc.usp.br}\urladdr{http://www.icmc.sc.usp.br/$\sim$tahzibi}

\title{Homoclinic tangency and variation of entropy}
 \begin{abstract}
% In this paper  we study the effect of a
%    homoclinic tangency in the variation of the topological entropy.
%    We prove that a diffeomorphism with a homoclinic tangency associated to a basic hyperbolic set with maximal
%    entropy is a point of entropy variation in the $C^{\infty}-$ topology. We also discuss variational problem in $C^1$- topology and the continuity of the topological entropy on tridimensional manifolds.

    In this paper we study the effect of a homoclinic tangency in the variation of the topological entropy.
    We prove that a diffeomorphism with a homoclinic tangency associated to a basic hyperbolic set with maximal entropy is a point of entropy variation in the $C^{\infty}$-topology. We also prove results about variation of entropy in other topologies and when the tangency does not correspond to a basic set with maximal entropy.
   We also show  an example of discontinuity  of the entropy among $C^{\infty}$ diffeomorphisms of  three dimensional manifolds.

\end{abstract}

\maketitle

%\vspace{2cm}

%*******************************************************************************************
%*************    First Section - Introduction   ********************************************
%*******************************************************************************************

\section{Introduction}
Topological entropy is one of the most important invariants of
topological conjugacy in dynamical systems. By the
$\Omega-$stability of Axiom A diffeomorphisms with no cycle
condition, it comes out that the entropy is a $C^1$-locally
constant function among such dynamics. We say that a diffeomorphism
$f$ is a point of constancy of topological entropy in $C^k$
topology if there exists a $C^k$-neighborhood $\mathcal{U}$ of $f$
such that for any diffeomorphism $g \in \mathcal{U}, h(g)=h(f).$
We also call a diffeomorphism as a point of  entropy variation if it is not a point of constancy.

In \cite{PS2}, Pujals and Sambarino proved that  surface
diffeomorphisms far from homoclinic tangency are the constancy
points of topological entropy in $C^{\infty}$ topology. In this
paper we address the reciprocal problem. We are interested
in the effect of a homoclinic tangency to the variation of the
topological entropy for a surface diffeomorphism.  Of course after
unfolding a homoclinic tangency, new periodic points will emerge,
but it is not clear whether they contribute to the variation of
the topological entropy. We mention that D\'{i}az-Rios \cite{diazrios} studied
unfolding of  critical saddle-node horseshoes and when the
saddle-node horseshoe is not an attractor they proved that the
entropy may decrease after the bifurcation. In our context, the tangency occurs
outside a basic hyperbolic set.

For Axiom A diffeomorphisms, by the spectral decomposition theorem
of Smale (see \emph{e.g.} \cite{Sm})), we have
$\Omega(f)=\Lambda_1 \cup \Lambda_2 \cup \dots \cup \Lambda_k$,
where each $\Lambda_i$ is a basic set, i.e,  an isolated
$f$-invariant hyperbolic set with a dense orbit. By the definition
of topological entropy we have
    $$
        h(f) = \max_{0\le i\le k}{h(f|_{\Lambda_i})}.
    $$
So, we conclude that there exists  a set  (at least one) which is ``\emph{responsible}"
for the topological entropy of an Axiom A. By this we mean,  there exists
some $k_0\in\{1,\dots,k\}$ such that $h(f)=h(f|_{\Lambda_{k_0}})$. Such $\Lambda_{k_0}$ is called responsible for entropy.

We consider  a class of diffeomorphisms on the frontier of Axiom A
systems which exhibit a homoclinic tangency  corresponding to a
periodic point of $\Lambda_{i_0}.$ We show that  the topological
entropy increases after small $C^{\infty}$ perturbations.

More precisely, consider a parametrized family $f_{\mu}: M \to M$
of diffeomorphisms of a closed surface $M$ unfolding generically (like in \cite{palistakens}) a
homoclinic tangency at $\mu =0$  where  $\Omega(f_0) = \Lambda_1
\cup \cdots \cup \Lambda_k \cup \mathcal{O}(q)$ where each $\Lambda_i$
is  an isolated hyperbolic set and $\mathcal{O}(q)$  is the unique homoclinic  tangency orbit
associated to a saddle fixed point $p$ of some $\Lambda_{i_0}.$

\begin{maintheorem} \label{t.responsible}
    Let $f_{\mu}$ be a one parameter family of $C^{2}$ surface
    diffeomorphism as above, then
if $\Lambda_{i_0}$ is responsible for the entropy,  $f_0$ is a
        variation point of the topological entropy in $C^{r}$ topology, for $2 \le r \le \infty$.
            \end{maintheorem}
      A  natural question is what happens if the tangency corresponds to a piece which is not responsible for the entropy. We prove that;

      \begin{maintheorem}  \label{t.nonresponsible1}
      Let $f_{\mu}$ be a one parameter family of $C^{2}$ surface
    diffeomorphism as above, then
    \begin{itemize}
    \item If $\Lambda_{i_0}$ is not responsible for the entropy, then $f_0$ is a constancy point of the topological entropy in the $C^{\infty}$ topology.
    \item $f_0$ is a point of constancy in $C^k-$topology ( $1 \leq k < \infty$) if $h_{top}(f) - h(f| \Lambda_{i_0}) > \alpha_k$ where $\alpha_k > 0$ is a constant depending on $f_0$.
    \end{itemize}
\end{maintheorem}

So the above theorems assert that, in the $C^{\infty}$ topology, if the tangency is at the ``correct" place (the responsible basic set), then the entropy varies and, if the tangency is at the ``wrong" place, then the entropy remains constant. It is interesting to note that if we consider $C^1$-topology then even a tangency ``in a wrong place"  may cause entropy variation:

\begin{maintheorem}\label{t.nonresponsible}
    There exists a diffeomorphism $f$of $S^2$ fixing a saddle with homoclinic transversal intersection and another one with homoclinic tangency such that $f$ {\bf is} a point of entropy variation in $C^1$-topology. \end{maintheorem}

By the existence of transverse homoclinic point in the example of above theorem the entropy of system is positive. Moreover the entropy is localized on the part of non wandering set far from homoclinic tangency. However, $C^1$-perturbations make the total entropy of the system increase.

%%However, it is interesting to know if it is possible to make such
%examples when the entropy of the piece corresponding to the tangency
%is very far from the topological entropy (See section \ref{yomdin}
%for a discussion and result of Yomdin in this direction).

We recall  a method for perturbation of surface dynamics with
homoclinic tangency, due to Newhouse, which is so called the
``Snake like" perturbation. Although after such perturbation the
non wandering set becomes richer, the topological entropy does not
necessarily increase. See theorem \ref{minhoca} for the relation
between an estimate of entropy after the perturbation and the
eigenvalues of the periodic point corresponding to the homoclinic
tangency.

\begin{theorem}[\cite{newhouseminhoca}]\label{minhoca}
    Let $p$ be a (conservative)  saddle point of a $C^1$-diffeomorphism
    $f$, such that $W^u(\mathcal{O}(p))$ is tangent to $W^s(\mathcal{O}
    (p))$ at some point. Given $\varepsilon>0$, for any
    neighborhood $\mathcal{N}$ of $f$ there exists $g\in\mathcal{N}$ such
    that
                $$
                    h(g) >
                    \frac{1}{\tau(p)}\log|\lambda(p)|-\varepsilon,
                $$
    where $\tau(p)$ is the period of $p$.
\end{theorem}
In the above theorem $\det(Df(p)) = 1$. However, if this is not the case, $|\lambda(p)|$ in the above theorem can be substitutes by $\min\{\lambda(p), \mu(p)^{-1}\}$ where $\lambda(p) > 1 > \mu(p)$ are the eigenvalues of $Df(p).$ 
As a corollary of continuity of topological entropy for the surface $C^{\infty}$ diffeomorphisms we conclude that:

\begin{maintheorem}\label{non.newhouse}
It is not possible to substitute $C^{\infty}$ instead of $C^{1}$ in the above theorem.
\end{maintheorem}

The continuity property of entropy is a challenging problem in smooth ergodic theory. Newhouse \cite{newhousecontinuity} and Yomdin \cite{yomdin} results give semi-continuity of entropy in $C^{\infty}$ topology and using a Katok's result \cite{katok.artigo} the entropy is continuous for $C^{\infty}$ surface diffeomorphisms.

Let us also mention a result of Hua, Saghin and Xia
\cite{huasaghinxia} where they prove that the topological entropy is
locally constant for some partially hyperbolic diffeomorphisms with
one dimensional central bundle. They have shown that for a large class of partially hyperbolic diffeomorphisms with bi-dimensional central foliations, the entropy varies continuously. The authors also claim that without the homological conditions the result is not true, and they exhibit examples in higher dimension ($\ge 4$) where the result fails without such hypothesis. So it arises a natural question about the continuity of the topological entropy in the $C^{\infty}-$ topology for partially hyperbolic systems defined on 3-dimensional manifolds.

Here we give an example of $3-$dimensional manifold diffeomorphism which is not the point of continuity of entropy in $C^{\infty}$-topology. However, this example is far from being partially hyperbolic.

\begin{maintheorem} \label{examplediscontinuity}
    There exists a diffeomorphism on $3$-dimensional ball which is a discontinuity point of the topological entropy in the $C^{\infty}$-topology.
\end{maintheorem}

%*******************************************************************************************
%*************    Second section - Main Ingredients   ************************************
%*******************************************************************************************

\section{Main Ingredients}

\subsection*{Topological Entropy}\label{sect.top.entropy}

Consider $\Sigma_N=\{1,\dots,N\}^{\mathbb{Z}}$ and the \emph{shift}
$\sigma:\Sigma_N\to\Sigma_N$  given by
$\sigma(\mathbf{x})=\mathbf{y}$ where $y_i=x_{i+1}$, $i\in \Z$.
For $A=(a_{ij})_{i,j=1}^{N}$ a square $0-1$-matrix of order $N$, the correspondent
\emph{subshift of finite type} is the restriction of $\sigma$ to
$\Sigma_A = \{ \mathbf{x}\in \Sigma_N \mid a_{x_i x_{i+1}}=1
\,\,\, \text{ for} \,\,\, i\in\Z\}$.

It is  well known (see \cite{robinson})   that  if $\sigma|_{A}:\Sigma_A\to \Sigma_A$ is a  subshift of finite type, then
        $$
            h(\sigma_A) = \log (\lambda_{max}),
        $$
        where $\lambda_{max}$ is the biggest eigenvalue of $A$ in
        modulus.

%The main properties of the topological entropy that we use in this
%text can be found in \cite{katok}.

%*******************************************************************************************
%*************     Homoclinic Markov                    ************************************
%*******************************************************************************************

\subsection*{$\Omega$-Homoclinic Explosions and Markov Partitions}

One of the important features of hyperbolic dynamics is the existence of Markov partitions.
 with rectangles of arbitrarily small
    diameter.

In particular, for basic set from the spectral decomposition of a
Axiom A diffeomorphism there are Markov partitions with
arbitrarily small rectangles and this implies conjugacy between $f|_{\Lambda}$ and a
subshift of finite type $\sigma|_{\Sigma_A}$.

To prove our first theorem, we focus on $\Omega-$explosion like in
the model in Palis-Takens result \cite{palistakens}, of course,
without any hypothesis on the fractal dimensions.

That is, we are considering a one parameter family $f_{\mu}$ where for the parameter $\mu=0$ the
nonwandering set $\Omega(f_0) = \Lambda_1 \cup  \Lambda_2 \cup \cdots \cup \tilde{\Lambda}_{i_0} \cup \dots \cup \Lambda_k$
such that $ \Lambda_i, i \neq i_0$ is a hyperbolic basic and
$\tilde{ \Lambda}_{i_0} =  \Lambda_{i_0} \cup \mathcal{O}(q)$
 where  $ \Lambda_k$ is a basic set and  $\mathcal{O}(q)$ is the orbit of a homoclinic tangency
associated with a saddle fixed point $p \in  \Lambda_k.$

For $\mu>0$ we can consider the basic sets $ \Lambda_i(\mu)$ as the
continuation of $\Lambda_i$. Thereby, we have  that
$\Lambda_i(\mu)$ is hyperbolic and $f_{\mu}|_{\Lambda_{i}(\mu)}$
is conjugated to $f_0|_{\Lambda_i}$. Then, we have
    $$
        h(f_{\mu}|_{\Lambda_{i}(\mu)}) = h(f_0|_{\Lambda_i})
    $$
for all $i=1,\dots,k$ and all $\mu$ positive or negative.

However, when we unfold the family $f_{\mu}$ new periodic points are created and the entropy of the
nonwandering sets may increase for positive parameters $\mu$. We
will see that, in fact, the entropy increases for small positive
parameters.
This can be shown by constructing a subsystem of $f_{\mu}$ with a dynamics richer than  $f_0|_{\Lambda_{i_0}}$.

To construct such a subsystem, we find a subset of $\Omega(f_{\mu})$
containing $\Lambda_{i_0}(\mu)$ using Markov partitions.
Take a parameter $\mu$ very close to $\mu=0$. Since $f_{\mu}$
unfolds generically, the map $f_{\mu}$ has transversal homoclinic
intersections close to $\mathcal{O}(q_0)$, the tangency orbit of
$f_0$. We have the situation represented below in the figure
\ref{figura1}.

%================================================================================
%                     FIGURA
%================================================================================

\begin{figure}[htb!]
  \centering
  \psfrag{A}{\footnotesize{$p$}}
  \psfrag{B}{\footnotesize{$q$}}
  \psfrag{C}{\footnotesize{$W^s(p)$}}
  \psfrag{D}{\footnotesize{$W^u(p)$}}
  \psfrag{E}{\footnotesize{$p_{\mu}$}}
  \psfrag{F}{\footnotesize{$q_{\mu}$}}
  \psfrag{G}{\footnotesize{$W^s(p_{\mu})$}}
  \psfrag{H}{\footnotesize{$W^u(p_{\mu})$}}
  \includegraphics[width=14cm]{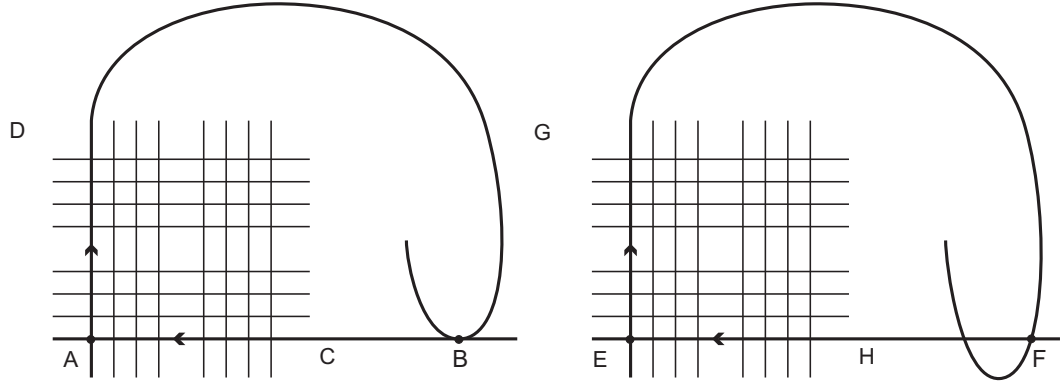}\\
  \caption{Unfolding of a homoclinic tangency close to $\mu=0$.}
  \label{figura1}
 \end{figure}

%================================================================================

Consider $q_{\mu}$ a transversal homoclinic intersection point
between $W^s(p_{\mu})$ and $W^u(p_{\mu})$ close to $q_0$ (the
tangency for $f_0$). Since $\Lambda_{{i_0}}(\mu)$ is hyperbolic and
maximal invariant set for $f_{\mu}$, there exist a isolating
neighborhood of $\Lambda_{{i_0}}(\mu)$, say $V_{{i_0}}$. Suppose that
$q_{\mu}\not\in V_{{i_0}}$. Moreover, we can use the Bowen's
construction of Markov partition \cite{bowen75}.  Consider
$\{R_1,\dots,R_s\}$ a Markov partition for $\Lambda_{{i_0}}(\mu)$ such
that
    $$
        \Lambda_{{i_0}}(\mu) = \bigcup_{j=1}^{s}R_j \subset V_{{i_0}}.
    $$
Furthermore, as $q_{\mu}\not\in V_{i_0}$ we have that a part of
$\mathcal{O}(q_{\mu})$ remains out of $V_{i_0}$. Take $N_1, N_2 \in
\mathbb{N}$ such that $f_{\mu}^{N_1}(q_{\mu})\in R_s$,
$f_{\mu}^{-N_2}(q_{\mu})\in R_1$ and
$f_{\mu}^j(q_{\mu})\not\in{\bigcup_{j=1}^{s}{R_j}}$ for
$j=-N_2+1,\dots,0,\dots,N_1-1$. In other words, $R_s$ is the
rectangle containing the first forward iterated of $q_{\mu}$ that
belongs to $V_{i_0}$, and $R_1$ is the rectangle containing the first
backward iterated of $q_{\mu}$ that belongs to $V_{i_0}$.

%================================================================================
%                     FIGURA
%================================================================================

\begin{figure}[htb!]
  \centering
  \psfrag{E}{\footnotesize{$p_{\mu}$}}
  \psfrag{F}{\footnotesize{$q_{\mu}$}}
  \psfrag{G}{\tiny{$f_{\mu}(q_{\mu})$}}
  \psfrag{H}{\tiny{$f^{2}_{\mu}(q_{\mu})$}}
  \psfrag{I}{\tiny{$f^{N_1}_{\mu}(q_{\mu})$}}
  \psfrag{J}{\tiny{$f^{-1}_{\mu}(q_{\mu})$}}
  \psfrag{K}{\tiny{$f^{-2}_{\mu}(q_{\mu})$}}
  \psfrag{L}{\tiny{$f^{-N_2}_{\mu}(q_{\mu})$}}
  \psfrag{M}{\footnotesize{$R_1$}}
  \psfrag{N}{\footnotesize{$R_s$}}
  \psfrag{O}{\footnotesize{$C$}}
  \psfrag{P}{\footnotesize{$f^{N_1}_{\mu}(C)$}}
  \psfrag{Q}{\footnotesize{$f^{-N_2}_{\mu}(C)$}}
  \includegraphics[width=14cm]{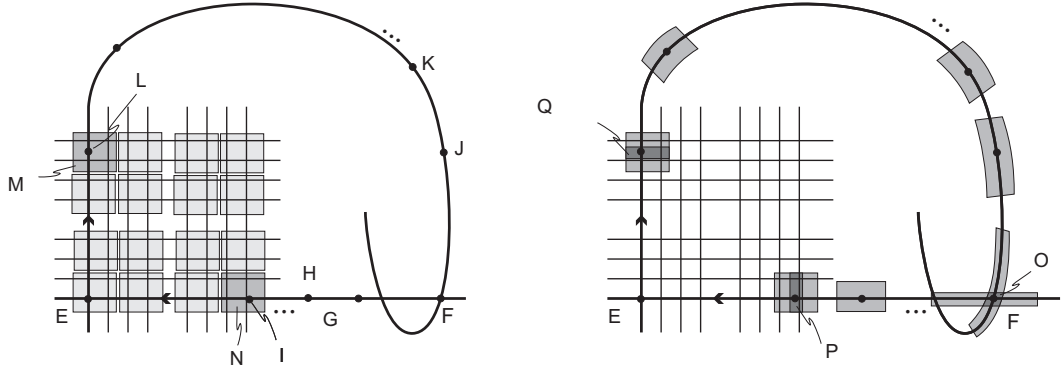}\\
  \caption{Construction of the Markov partition.}
  \label{figura2}
 \end{figure}

%================================================================================

Given the Markov partition for $\Lambda_{i_0}(\mu)$, we extend
it for a larger set which contains
$\Lambda_{i_0}(\mu)\cup\mathcal{O}(q_{\mu})$ by constructing
other rectangles containing
$\{f_{\mu}^{-N_2 + 1}(q_{\mu}),\dots,q_{\mu},\dots,f_{\mu}^{N_1 + 1}(q_{\mu})\}$
in the following way: if we iterate $R_1$ under $f^{N_2}$, we get a
narrow strip around $W^u(p_{\mu})$ containing $q_{\mu}$. And if we
iterate $R_s$ under $f_{\mu}^{-N_1}$, we get a narrow strip around
$W^s(p_{\mu})$ containing $q_{\mu}$. We know that $W^s(p_{\mu})$ and
$W^u(p_{\mu})$ have transversal intersection on $q_{\mu}$. As we
could take the diameter of the partition small enough, it comes out
that $f_{\mu}^{-N_1}(R_s)$ and $ f_{\mu}^{N_2}(R_1)$ are   transversal.
Let $C:= f_{\mu}^{-N_1}(R_s) \cap f_{\mu}^{N_2}(R_1)$. It is clear that $C$
is disjoint from $\bigcup_{i=1}^s R_i$ and contains $q_{\mu}.$

Note that $f_{\mu}^{N_1}(C)$ is a vertical strip of full height
contained in $R_s$ and $f_{\mu}^{-N_2}(C)$ is a horizontal strip of
full weight contained in $R_1$. Consider the disjoint sets $S_i$
defined as
    $$
        S_j=f_{\mu}^{-N_2+j}(C)
    $$
for $j = 1,2,\dots,N_2,N_2+1,\dots,N_1+N_2-1.$ Note that
$S_{N_2}=C$. Now denote $\ell=N_1+N_2-1$ and consider
$\mathcal{P}=\{R_1,\dots,R_s,S_1,\dots,S_\ell\}$ and
    $$
        R=\bigcup_{i=1}^{s}R_i\cup\bigcup_{j=1}^{\ell}S_j.
    $$
So $\displaystyle{\Lambda_R = \bigcap_{n\in \Z}f_{\mu}^{n}(R)}$ is an
isolated hyperbolic set such that $\Lambda_{i_0}(\mu)\subset \Lambda_R \subset \Omega(f_ {\mu})$.
The desired subsystem is the restriction $f_{\mu}:\Lambda_R\to\Lambda_R$.

\begin{lemma}
     $\mathcal{P}$ is Markov partition for $\Lambda_R$.
\end{lemma}
\begin{proof}
    We already know that all $R_i$'s satisfy the Markov property. It remains to verify the Markov  property for $S_j$'s. By
    construction we have all $R_i$'s and $S_j$'s pairwise disjoint.
    Furthermore $f_{\mu}(S_j)=S_{j+1}$ for $j=1,\dots,\ell-1$.
    In particular, we have $f_{\mu}(S_{\ell})\subset R_s$ is a vertical
    strip of full height. So
        \begin{align*}
            & f_{\mu}(S_{\ell})\cap R_s\neq\varnothing,\\
            & f_{\mu}(S_{\ell})\cap
            R_i = \varnothing, \mbox{ for } i=1,\dots,s \mbox{ and
            }\\
            & f_{\mu}(S_{\ell})\cap S_j = \varnothing, \mbox{ for }
            j=1,\dots,\ell.
        \end{align*}

    On the other hand, only $R_1$ has image by $f_{\mu}$ that intersects some
    $S_j$. In fact,
        $$
            f_{\mu}(R_1)\cap S_1\neq \varnothing \mbox{ and }
            f_{\mu}(R_1)\cap S_j =\varnothing, \mbox{ for } j=2,\dots,\ell.
        $$
    Note that since $f_{\mu}^{-1}(S_1)\subset R_1$ is a horizontal strip
    of full weight in $R_1$, then $S_1=f_{\mu}(f_{\mu}^{-1}(S_1))\subset
    f_{\mu}(R_1)$. So, $f_{\mu}(R_1)\cap S_1\neq\varnothing$ and, by the
    construction of $S_1$, this intersection satisfies the transversality condition
    of Markov partitions. Thus $\mathcal{P}$ is a Markov partition
    for $\Lambda_R$.
\end{proof}

We can associate to $f_{\mu}:\Lambda_R\to\Lambda_R$  a
subshift of finite type as follows. We consider the Markov partition
$\mathcal{P}=\{P_1,\dots,P_{s+\ell}\}$ as above and we define a
transition matrix $A_{\mu}=(a_{ij})_{(s+\ell)\times(s+\ell)}$ for
$f_{\mu}$ taking
    $$
    a_{ij}=
        \left\{
        \begin{array}{ll}
            1, & \hbox{if $f_{\mu}(P_i)\cap P_j\neq \varnothing$;} \\
            0, & \hbox{if $f_{\mu}(P_i)\cap P_j = \varnothing$}
        \end{array}
        \right.
    $$
for $i,j\in\{1,\dots,s+\ell\}$. In this way we obtain a topological conjugacy
between the systems $f_{\mu}:\Lambda_R\to\Lambda_R$ and the
subshift of the finite type
$\sigma_{A_{\mu}}:\Sigma_{A_{\mu}}\to\Sigma_{A_{\mu}}$, where
$\Sigma_{A_{\mu}}\subset\Sigma_{s+\ell}$. The transition matrix
$A_{\mu}$ has the following form
    $$
        a_{ij}=
            \left\{
                \begin{array}{ll}
                    H_{ij}, & \hbox{if $1\le i,j\le s$;} \\
                    1, & \hbox{if $i=1,j=s+1$ or $i=s+\ell, j=s$;} \\
                    1, & \hbox{if $j=i+1$ for $s+1\le i\le s+\ell-1$;} \\
                    0, & \hbox{in other cases.}
                \end{array}
            \right.
    $$
where $H_{\mu}=(H_{ij})_{s\times s}$ is the transition matrix of
$f_{\mu}:\Lambda_{i_0}(\mu)\to\Lambda_{i_0}(\mu)$ which is
irreducible, because $f_{\mu}|_{\Lambda_{i_0}(\mu)}$ is topologically
transitive (see the next section).

%%*******************************************************************************************
%%*************   Section four- Proof of Result         ************************************
%*******************************************************************************************

\section{Proof of Theorem \ref{t.responsible}}

Let $f_{\mu}$ be the one parameter family as in the theorem \ref{t.responsible}.
In the previous section we constructed a Markov partition for the subsystem
$f_{\mu}|_{\Lambda_R}$, for $\mu\ge 0$. By means of this Markov partition, one may
give a conjugacy between such invariant subsystem $f_{\mu}|_{\Lambda_{i_0}(\mu)}$ and the dynamics
of a subshift of finite type.

Let $A_{\mu}$ be the transition matrix of $f_{\mu}|_{\Lambda_R}$, for $\mu>0$
small enough. Recall that $h(f_{\mu}) = \log
\lambda_{\mu}$ where $\lambda_{\mu}$ is the largest eigenvalue of
$A_{\mu}$. By construction of Markov partition in the previous
section, we conclude that

\begin{equation}
    A_{\mu} =
    \left(
    \begin{array}{cc}
        \left[
          \begin{array}{c}
            \begin{minipage}[c]{2.2cm}
                \vspace{0.8cm}
                    \begin{center}
                        $H_{\mu}$
                    \end{center}
                \vspace{0.6cm}
            \end{minipage}
          \end{array}
        \right]
  &
        \left[
          \begin{array}{cccccc}
            1 & 0 & 0 &  \cdots & 0 \\
            0 & 0 & 0 &  \cdots & 0 \\
            \vdots & \vdots & \vdots & \ddots & \vdots \\
            0 & 0 & 0 & \cdots & 0 \\
          \end{array}
        \right]
  \\
  &
  \\
        \left[
          \begin{array}{cccc}
            0 & 0 & \cdots & 0 \\
            0 & 0 & \cdots & 0 \\
            \vdots & \vdots & \ddots & \vdots \\
            0 & 0 & \cdots & 0 \\
            0 & 0 & \cdots & 1 \\
          \end{array}
        \right]
  &
        \left[
          \begin{array}{ccccc}
            0 & 1 & 0 & \cdots & 0 \\
            0 & 0 & 1 & \cdots & 0 \\
            \vdots & \vdots & \vdots & \ddots & \vdots \\
            0 & 0 & 0 & \cdots & 1 \\
            0 & 0 & 0 & \cdots & 0 \\
          \end{array}
        \right]
    \end{array}
    \right).
%  =
%    \left(
%      \begin{array}{cc}
%        H_{\mu} & I_1 \\
%        I_2 & I_3 \\
%      \end{array}
%    \right)
\label{matrix.A}
\end{equation}

The following proposition asserts that the largest eigenvalue of the matrix $A_{\mu}$ is strictly
bigger than the largest eigenvalue of the matrix $A_0=H_0$. From this proposition we can conclude
that the entropy of the system $f_{\mu}|_{\Lambda_{i_0}(\mu)}$ is bigger than the entropy of $f_{0}|_{\Lambda_{i_0}}$.

\begin{proposition}\label{t.lambda}
    Let $A_{\mu}$ as defined above. If $\lambda_{\mu}$ is the largest eingenvalue of $A_{\mu}$ in modulus,
    then for any $\mu>0$ near to zero, $\lambda_{\mu}>\lambda_{0}$.
\end{proposition}

\begin{proof}
    To proof the proposition we use  Perron and Frobenius theorem.

\begin{theorem}[Perron-Frobenius, \cite{gantmacher}]
    Every non-negative $s\times s$ matrix $A$ has a non-negative eigenvector,
    $AU = \lambda U$, with the property that the associated $\lambda$ is equal to the
    spectral radius $|\lambda|_{max}$. If the matrix $A$ is irreducible, then there
    is just one non-negative eigenvector up to multiplication by positive constant,
    and this eigenvector is strictly positive.
    Furthermore, the maximal eigenvalue $\lambda'$ of every principal
    minor (of order less than $s$) of $A$ satisfies $\lambda'\le
    \lambda$. If $A$ is irreducible, then $\lambda' \lneqq \lambda$.
\end{theorem}

    To use this theorem we need the matrix
    $A_{\mu}$ be irreducible, that is, for any pair $i, j$ there is some power $n(i,
    j)$ of $A_{\mu}$ such that $A_{\mu}^{n(i,j)} > 0$.
    
    % By
    %the definition of the transition matrix $A_{\mu}$, we get a
   % characterization of irreducibility using Markov partitions.

%\begin{lemma}
%    The transition matrix $A_{\mu}$ for a Markov  partition $\mathcal{P}=\{
   % P_i\}$ is irreducible if, and only if, for each pair $i,j$ there
    %exists $n=n(i,j)$ such that $f^n(P_i)\cap P_j\neq \varnothing$.
%\end{lemma}
%\begin{proof}
   % The irreducibility is clear for the above  transition matrix.
%\end{proof}

 It is easy to see that $A_{\mu}$ is irreducible.
    Indeed, irreducibility is satisfied by the rectangles $R_i$'s
    because the system $f_{\mu}|_{\Lambda_{i_0}(\mu)}$ is transitive.
    Since for each $S_j$, the iterated $f^{\ell-j}(S_j)$ intersects $R_s$
    and the iterated $f^{j}(R_1)$ intersects $S_j$, we obtain
    the desired property for all elements of $\mathcal{P}$.

    Now we can apply the Perron-Frobenius Theorem to the sub-matrix $A_{\mu,1}$ of the
    irreducible transition matrix $A_{\mu}$, obtained by excluding the last line and the last
    column of $A_{\mu}$. So we obtain that the largest eigenvalue $\lambda_{\mu}$ of
    $A_{\mu}$ is strictly bigger than the largest eigenvalue $\lambda_{\mu,1}$ of
    $A_{\mu,1}$. Even though $A_{\mu,1}$ is not necessarily an irreducible matrix, we can use
    the Perron-Frobenius theorem again to the sub-matrix
    $A_{\mu,2}$, with largest eigenvalue  $\lambda_{\mu,2}$ and obtain that
    $\lambda_{\mu,2}\le \lambda_{\mu}$. We repeat this process to
    obtaining the sub-matrix $H_{\mu}$, whose largest eigenvalue
    $\lambda_{\mu,\ell}$ is equal to $\lambda_0$, because
    the systems $f_{\mu}|_{\Lambda_{i_0}(\mu)}$ and $f_0|_{\Lambda_{i_0}}$ are topologically
    conjugated. Thus we have $\lambda_{\mu}\gneqq
    \lambda_0$.\hfill$\square$\newline

    To conclude the proof of Theorem
    \ref{t.responsible} observe that for all $C^2$-neighbor\-hood
    $\mathcal{V}$ of $f=f_0$ we can take $f_{\mu}$ with $\mu$ very close
    to $0$ such that $f_{\mu}\in\mathcal{V}$ and the Proposition
    \ref{t.lambda} holds. So, since $\Lambda_{i_0}$ is responsible for the
    entropy of $f_0$, $h(f_{\mu}) \ge h(f_{\mu}|_{\Lambda_{i_0}(\mu)}) >
    h(f_0|_{\Lambda_{i_0}})=h(f_0)$. Then $h(f_{\mu})\neq h(f_{0})$ and thus
    $f_0$ is a point of entropy variation.

\end{proof}

\section{ Proof of Theorem \ref{t.nonresponsible1}}

Here we recall some results of Yomdin \cite{yomdin} and Newhouse \cite{newhousecontinuity} for the calculation of the defect of continuity of the entropy
function in the space of $C^k$-diffeomorphisms. We use their result for maximal invariant subsets of dynamics.

Let $\Lambda_f = \bigcap_{n \in \mathbb{Z}} f^n(U)$ be an isolated maximal invariant subset and $r(n, \epsilon, f)$ denote the maximal cardinality of an $(n, \epsilon)-$separated subset of $\Lambda_f.$ Denote $r(\epsilon, f) = \limsup_{n \rightarrow \infty} \frac{1}{n} r(n, \epsilon, f).$ By definition $h(f | \Lambda_f) = \lim_{\epsilon \rightarrow 0} r(\epsilon, f).$ We need to find an upper bound for $h(f | \Lambda_f) - r(\epsilon, f).$ Local entropy turns out to be such upper bound. The local entropy can be defined for any invariant measure as follows: Take $\Lambda $ any compact subset of $M$ and for any $x \in \Lambda, \epsilon > 0$ let $$W^s(x, n, \epsilon) := \{y \in M, d(f^i(x), f^i(y)) \leq \epsilon \quad  \text{for} \quad i \in [0, n)\}.$$ Let $r(n ,\delta, \epsilon, \Lambda, x) = max \,\, Card (F)$ where $$F \subset \Lambda \cap W^s(x, n, \epsilon)$$ is $(n, \delta)-$separated. Taking $r(n, \delta, \epsilon, \Lambda) = \sup_{x \in \Lambda} r(n ,\delta, \epsilon, \Lambda, x)$ we define
$$
 r(\epsilon,  \Lambda) = \lim_{\delta \rightarrow 0} \limsup_{n \rightarrow \infty} \frac{1}{n} \log
r(n, \delta, \epsilon, \Lambda)$$
and for any invariant measure $\mu$ we define $h_{loc}^{\mu} (\epsilon, f) = \lim_{\sigma  \rightarrow 1} \inf_{\mu(\Lambda) > \sigma} r(\epsilon, \Lambda).$ Finally define $h_{loc} (\epsilon, f) = \sup_{\mu} h_{loc}^{\mu} (\epsilon, f).$
Newhouse proved (Theorem 1 of \cite{newhousecontinuity}) that:
$$
 h (f) \leq r(\epsilon, f) + h_{loc}(\epsilon, f).
$$
 Take any $g$ close to $f$  define $\Lambda_g := \bigcap_{n \in \mathbb{Z}} g^{n} (U)$.
\begin{proposition} \label{adaptation}
The map $g \rightarrow h(g| \Lambda_g)$ is upper semicontinuous in $C^{\infty}-$topology.
\end{proposition}

\begin{proof}
The proof is an adaptation of Newhouse proof.  We use that for any  $g, h(g|\Lambda_g) \leq r(\epsilon, g|\Lambda_g) + h_{loc}(\epsilon, g|\Lambda_g)$ repeating the arguments of Newhouse. For any $C^{\infty}$ diffeomorphism $f$ and $g$ close enough in $C^{\infty}-$topology, the local entropy of $g$ is small (This comes from Yomdin result \cite{yomdin} ) and $r(\epsilon, g|\Lambda(g))$ is also upper semicontinuous. More precisely, Let $\mathcal{A}$ be any cover of $U$ with $diam(U) \leq \epsilon$ then $r(\epsilon, g|\Lambda_g) \leq h(\mathcal{A}, g)$ where $h(\mathcal{A}, g)$ is the entropy of the covering. It is easy to see that $g \rightarrow h(\mathcal{A}, g)$ is upper semicontinuous, because it is infimum of upper semicontinuous functions.

A similar proof to Yomdin Theorem \cite{yomdin} implies
\begin{proposition} \label{defectck}
    Let $f:M\to M$ be $C^k$ e $g_n\to f$ in $C^k$ topology then,
        \begin{equation}
            \limsup_{n \rightarrow \infty} h(g_n |\Lambda(g_n))\le h(f | \Lambda(f))+\frac{2m}{k}R(f),
        \end{equation}
    where $k\ge 1$, $m=\dim M$ e $\displaystyle{R(f)=\lim_{n\to\infty}{
    \frac{1}{n}\log\max_{x\in M}\|Df^n(x)\|}}.$
\end{proposition}

%\begin{question}\label{questao}
% Let $f$ be a $C^1-$diffeomorphism with $\Omega(f) = \Lambda_1 \cup \cdots \Lambda_k
%\cup \mathcal{O}(q)$ where $\Lambda_i$ are basic pieces and
%$\mathcal{O}(q)$ is the homoclinic tangency corresponding to a piece
%not responsible for the entropy, for instance, $\Lambda_1$.  Is
%there $\kappa > 0$ such that if $|h(f) - h(f| \Lambda_1)| > \kappa$
%then $f$ is a point of constancy in the $C^1$ topology?
%\end{question}

%==========================================================================================
%===============     relation to Yomdin if M = S^2   ======================================
%==========================================================================================

%\subsection{The attractor case}
%Let $f:M\to M$ be a diffeomorphism, with $M$ a compact bidimensional manifold,
%satisfying $\Omega(f)=\Lambda_1\cup\Lambda_2\cup\dots\cup\Lambda_k\cup\mathcal{O}(q)$, where $q$ is the unique point of tangency between the stable and unstable manifolds of a fixed (or periodic) point $p\in\Lambda_{1}$.
%

The above semi continuity results are key points of the proof of Theorem \ref{t.nonresponsible1}.
So we state our setting again:
   $f:M \to M$ is a diffeomorphism (Here $f$ stands for $f_0$ in Theorem \ref{t.nonresponsible1}.)  such that the homoclinic tangency $\mathcal{O}(q)$ corresponds to the \textbf{non} responsible basic set. The nonwandering set $\Omega(f) = \bigcup_{i=1}^{k} \Lambda_i \cup \mathcal{O}(q)$ where $q$ is a tangency corresponding to $\Lambda_{i_0}.$

Shub in \cite{shub.stabilite} defined a filtration adapted to a homeomorphism $f:M\to M$ as a sequence $\varnothing = M_0 \subset M_1\subset \dots \subset M_k = M$, where each $M_i$ is a compact $C^{\infty}-$submanifold with boundary of $M$ such that
\begin{itemize}
  \item $\dim M_i = \dim M$;
  \item $f(M_i)\subset int(M_i)$.
\end{itemize}

Given a filtration $\mathcal{M}$ adapted to $f$, $K_{i}^f(\mathcal{M}) = \bigcap_{n\in \Z}f^n(M_{i}\setminus M_{i-1})$ is the maximal $f$-invariant subset of $M_{i}\setminus M_{i-1}$, which is compact. Furthermore, we denote $K^f(\mathcal{M})=\bigcup_{i=1}^k K_{i}^f(\mathcal{M})$ and for the nonwandering set $\Omega(f)$ we have
$$
    \Lambda_{i}=\Omega(f)\cap(M_{i}\setminus M_{i-1}).
$$

\begin{proposition}[\cite{shub.stabilite}]\label{shub}
    Let $\mathcal{M}$ be a filtration adapted to $f$ and $U$ a neighborhood of $K^{f}(\mathcal{M})$. Then there exists a $C^0$-neighborhood $\mathcal{U}$ of $f$ in the space of homeomorphisms on $M$ such that, for each $g\in \mathcal{U}$, $\mathcal{M}$ is a filtration adapted to $g$ and $K^g(\mathcal{M})$ is contained in $U$. Moreover, taking $U_{i}=(M_{i}\setminus M_{i-1})\cap U$, we can choose a neighborhood of $\mathcal{U}$ such that $K_{i}^{g}(\mathcal{M})\subset U_{i}$.
\end{proposition}

Suppose that $\Lambda_{i_0}$ is the piece with external homoclinic tangency.  By Palis-Takens \cite{palistakens}, for each diffeomorphism with a homoclinic tangency as above, there exists a filtration $\mathcal{M}$ for the decomposition $\Omega(f)=\Lambda_1\cup\Lambda_2\cup\dots\cup\Lambda_k\cup\mathcal{O}(q)$ such that
\begin{itemize}
  \item[(i)] $\Lambda_i = \bigcap_{n\in \Z}f^n(M_i \setminus M_{i-1})$, for $i\neq i_0$;
  \item[(ii)] $\Lambda_{i_0}\cup \mathcal{O}(q) = \bigcap_{n\in \Z}f^n(M_{i_0} \setminus M_{i_0 -1})$.
\end{itemize}

Consider $g$ a $C^{\infty}-$perturbation of $f$. As $\mathcal{O}(q)$ is the unique orbit of tangency and all pieces $\Lambda_i, i \neq i_0$ are hyperbolic sets then $h (f | \Lambda_i) = h( g|\Lambda_i), i \neq i_0$   and
using Proposition \ref{adaptation} we obtain that $h(g| \Lambda_{i_0} (g) ) \leq  h(f | \Lambda_{i_0} (f)) + \epsilon$ where $\epsilon \rightarrow 0$ as $g$ converges to $f$ and as $h(f | \Lambda_{i_0} (f)) < h(f)$ we conclude that $h(f) = h(g).$

To prove the second item of Theorem (in $C^k-$topology), we use Proposition \ref{defectck}. It is enough to consider $\alpha_k = \frac{2m}{k} R(f)$ and using filtrations as the proof of the first item of Theorem.

    %\hfill$\square$\newline
\end{proof}

%%*******************************************************************************************
%%*************   Section - Proof of Result 2       ************************************
%********************************************************************************************

\section{Proof of theorem \ref{t.nonresponsible}}
To prove  Theorem \ref{t.nonresponsible} we
construct a system with a horseshoe and a homoclinic tangency
corresponding to a hyperbolic fixed point outside the horseshoe.
Then we perturb the system in a small neighborhood of the tangency to
create a transversal intersections (using
$C^1$-perturbations ``Snake like" as in Newhouse
\cite{newhouseminhoca}) to obtain a new system with larger topological
entropy.

Consider the system $f$ on the sphere $\mathbb{S}^2$ whose orbits
follow the meridians from $a$ (the North Pole) to $b$
(the South Pole). Suppose that the system has a transverse homoclinic point and a
homoclinic loop in two disjoint regions. These regions are
sorrounded by meridians. See the Figure \ref{exemplo01}. Indeed, fistly consider a system with two homoclinic loop. By making an small perturbation  on one of the homoclinic loops one obtain a transverse homoclinic orbit and consequently horseshoes. 

\begin{figure}[htb!]
 \begin{minipage}[b]{0.50\linewidth}
  \centering
  \psfrag{a}{$p_{\infty}$}
  \psfrag{b}{$p_0$}
  \psfrag{c}{$\Gamma$}
  \psfrag{d}{$p$}
  \includegraphics[width=4.0cm]{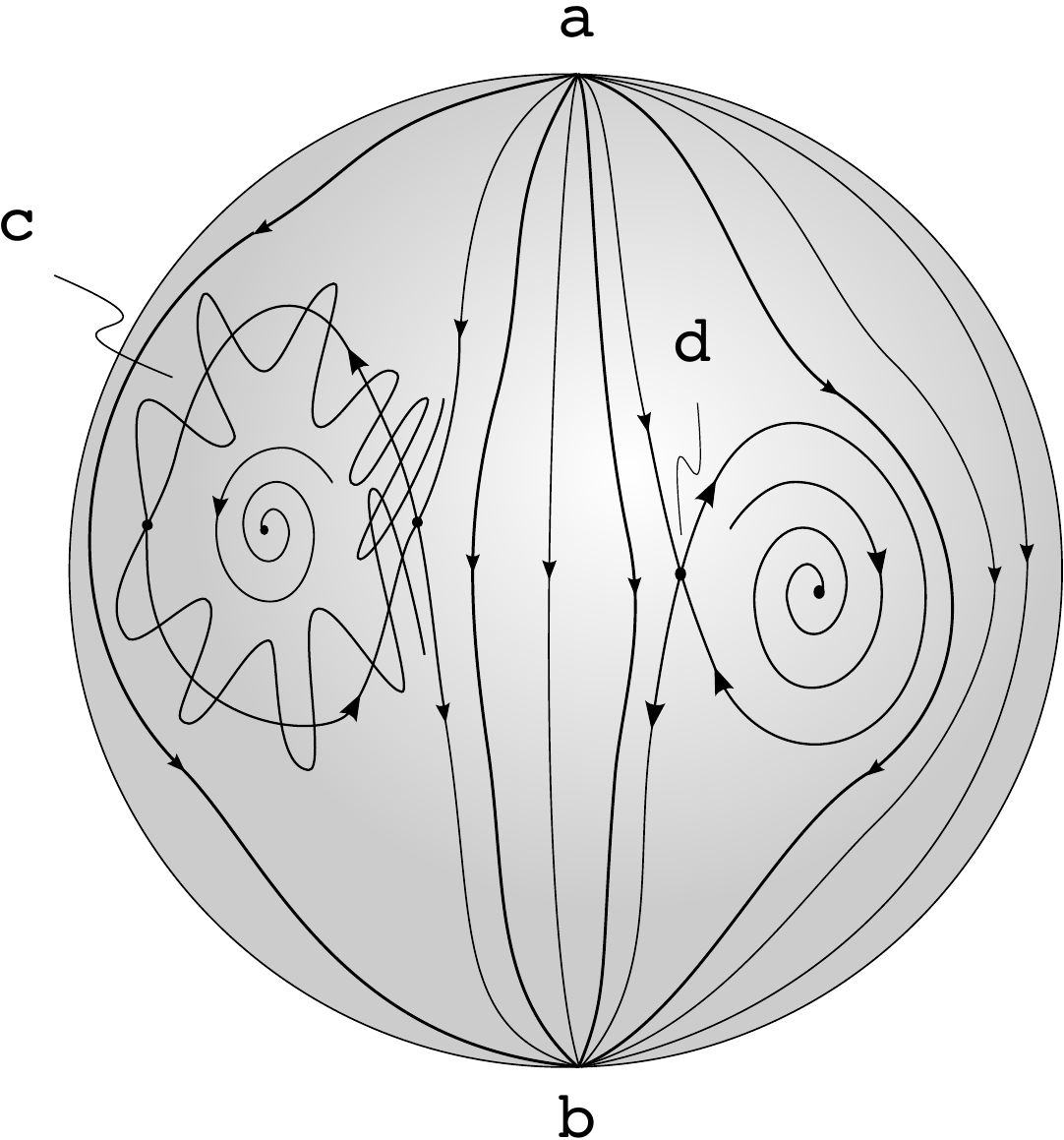}\\
  \caption{System with a horseshoe and a homoclinic loop.}
  \label{exemplo01}
\end{minipage} \hfill
\begin{minipage}[b]{0.48\linewidth}
  \centering
  \psfrag{a}{$p_{\infty}$}
  \psfrag{b}{$p_0$}
  \psfrag{c}{$Q$}
  \includegraphics[width=5cm]{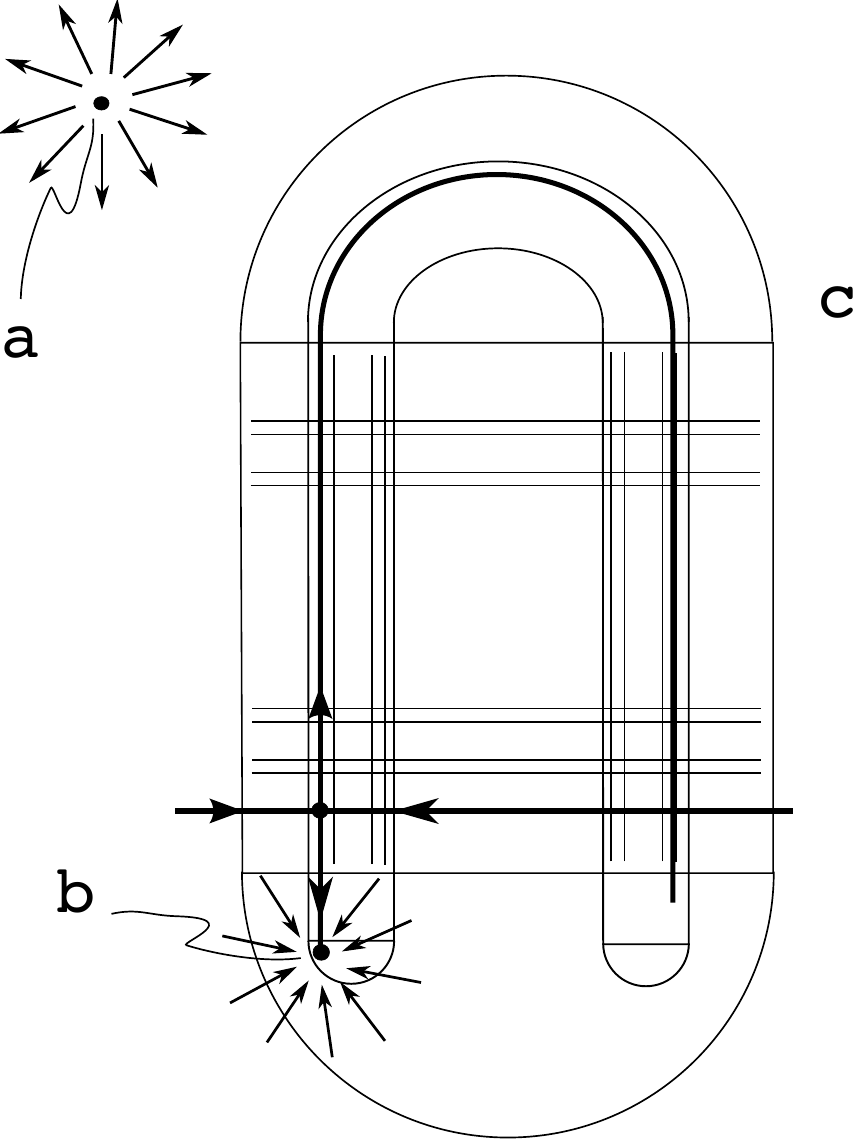}\\
  \caption{Region of the horseshoe $\Gamma$.}
  \label{exemplo03}
 \end{minipage}
\end{figure}

We suppose this homoclinic loop is associated to a fixed hyperbolic
point $p$ which has derivative with eigenvalues $\lambda(p)=3$ and
$\lambda(p)^{-1}=3^{-1}$. The horseshoe $\Gamma$ in the first region
is a two legs horseshoe, $p_{\infty}$ is the source which send the
orbits to a topological disc $Q$ whose interior is a trapping
neighborhood for $\Gamma$. The sink of this horseshoe coincides with
$p_0$. See the figure \ref{exemplo03}.

Thus the nonwandering set $\Omega(f)$ consists of three sinks, a
source, a (isolated) horseshoe and a hyperbolic point on the
homoclinic loop. Then, the topological entropy of $f$ is
$h(f|_{\Gamma})=\log 2$. Observe that this horseshoe is responsible 

Now we perturb $f$ in the $C^1$-topology to obtain a new system $g$,
breaking the homoclinic loop.
 Here we have an interval of tangency corresponding to the saddle $p$ and can apply Theorem \ref{minhoca}. 
By theorem \ref{minhoca} we obtain a diffeomorphism $g$ such that for small $\epsilon$,  $h(g|_{\Lambda}) \geq
 \log|\lambda(p)|-\varepsilon.$ As the tangency for $f$ is associated to a fixed point and $h(g) > h(f)$ the proof is complete.

 %Note that $g$ is different from $f$ just inside a region delimited
%by two meridians. Then, the entropy of $g$ restricted to $\Gamma$
%coincides with the entropy of $f$ restricted to $\Gamma$. Thus,
%the system $g:\mathbb{S}^2\to\mathbb{S}^2$ has
%$\overline{H(p,g)}\supset \Lambda$ as the responsible basic set.
%\hfill $\square$

\section{Proof of Theorem \ref{non.newhouse}}
It is well known that topological entropy function $f \rightarrow h_{top}(f)$ is a continuous function in $C^{\infty}$ topology for systems defined on a bi-dimensional compact manifold. Using this observation, our example in the proof of the second  statement of theorem \ref{t.responsible} shows that the Newhouse perturbation Theorem \ref{minhoca} can not be applied in $C^{\infty}$ topology.

\section{Lack of lower semi continuity}\label{exemplo.descontinua} 
We exhibit an example of discontinuity (lower semicontinuity) of the topological entropy in dimension three in $C^{\infty}$ topology 
. Observe that in $C^{\infty}$ toplogy the topological entropy if upper semi continuous.

\begin{theorem} \label{examplediscontinuity}
    There exists a diffeomorphism on the closed $3$-dimensional ball which is a discontinuity point of the topological entropy in the $C^{\infty}-$topology.
\end{theorem}

\begin{proof}
Consider a horseshoe map $f:D\to D$ of class $C^{\infty}$, where $D$ is the unitary disc of $\R^2$ as the Picture \ref{age.no.disco}. We divide it in three parts cut by the lines $y=\frac{1}{3}$ and $y=-\frac{1}{3}$.

\begin{figure}[h]
  \centering
  \includegraphics[width=6cm]{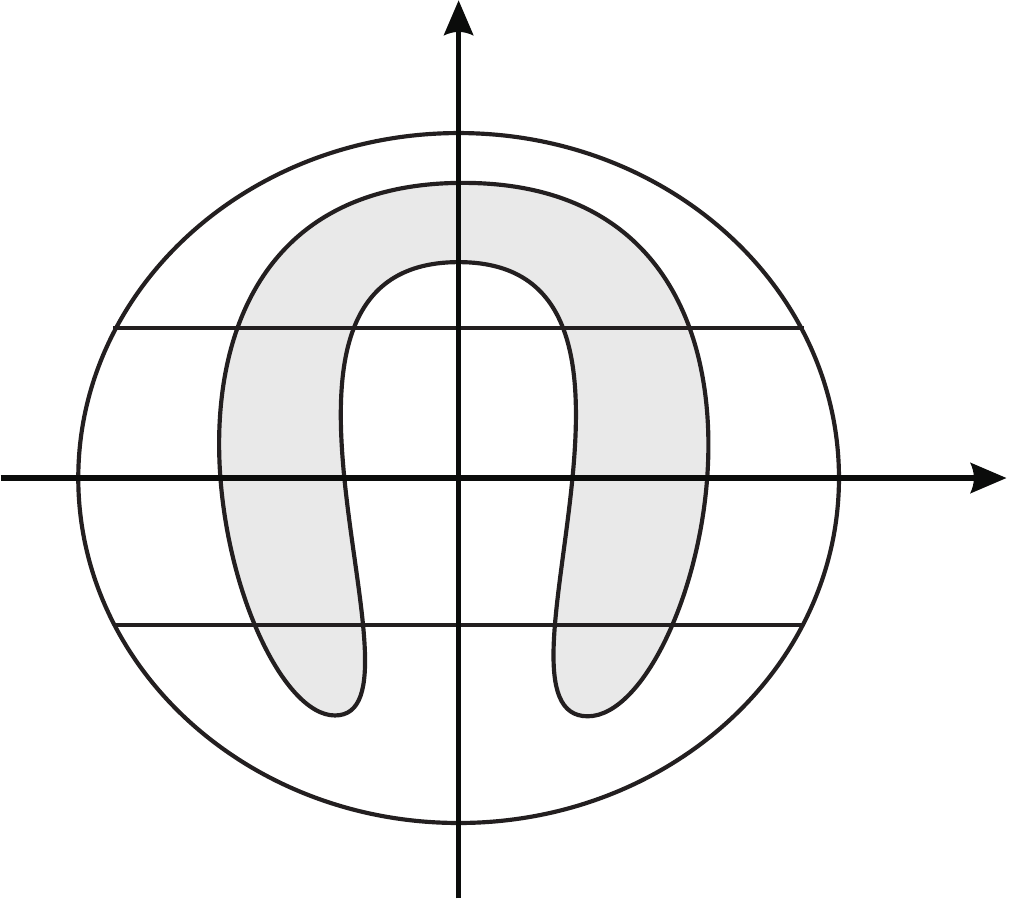}\\
  \caption{Horseshoe map on the disc.}
  \label{age.no.disco}
\end{figure}

We will now construct an isotopy which deforms the horseshoe map to a contraction. Let us produce a family $f_t$ as follows. First of all consider the maps $\alpha_t : \R^2\to \R^2$ given by $\alpha_{t}(x,y)= \Big( \frac{2}{3}x, (1-t)y -\frac{5}{3}t \Big)$, with $\varepsilon_0\le t\le 1$, where $\varepsilon_0>0$ is small enough such that the range $\alpha_{\varepsilon_0}(D)$ is under the $x-$axis. Now define $f_t = f\circ \alpha_t$ and note that for $t_0=\varepsilon_0$, the map $f_{t_0}$ is a contraction with a unique fixed point, which is an attractor. We can assume, unless reparametrization, that the family $f_t$ is defined for $t$ in $t\in[0,1]$ and that $f_0=f$. Extend this family for $t\in[-1,1]$ making $f_t=f_{-t}$ when $t\in [-1,0)$. Moreover, for each $t\in(-1,1)$ contract the domain of each $f_t$ is defined by decreasing the radius of $D$ to $\sqrt{1-t^2}$, i.e., define a map $C_t:D\to D_t$, where $D_t=\{(x,y):x^2+y^2\le1-t^2\}$, given by $C_t(x,y)=\big((\sqrt{1-t^2})x,(\sqrt{1-t^2})y\big)$. Thus for each $t\in(-1,1)$ define $\tilde f_{t}=C_t^{-1}\circ f_t\circ C_t$.

\begin{figure}
  \centering
  \includegraphics[width=11cm]{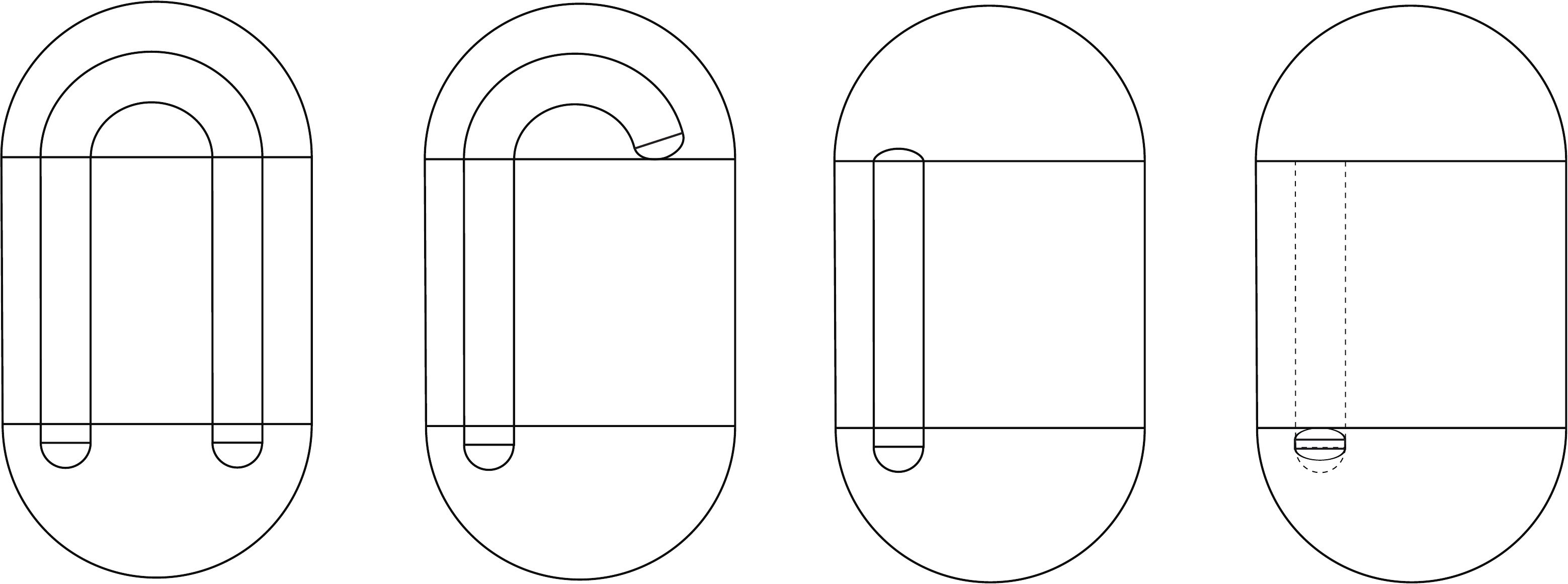}\\
  \caption{Family $f_t$ on the disc.}
  \label{deformacao}
\end{figure}

Now construct a new map $G$ defined on the $3$-dimensional disc $\mathbb{D}^3=\{(x,y,z)\in \R^3 : x^2+y^2+z^2\le1\}$ in the following way: far each $2$-dimensional disc $\mathbb{D}^2_{t} = \{(x,y,z)\in\mathbb{D}^3 : z=t \}$, for $t\in(-1,1)$, define $F(x,y,z=t)=(\tilde f_t(x,y),t)$.

\begin{figure}[htb!]
 \begin{minipage}[b]{0.50\linewidth}
  \centering
  \includegraphics[width=5.5cm]{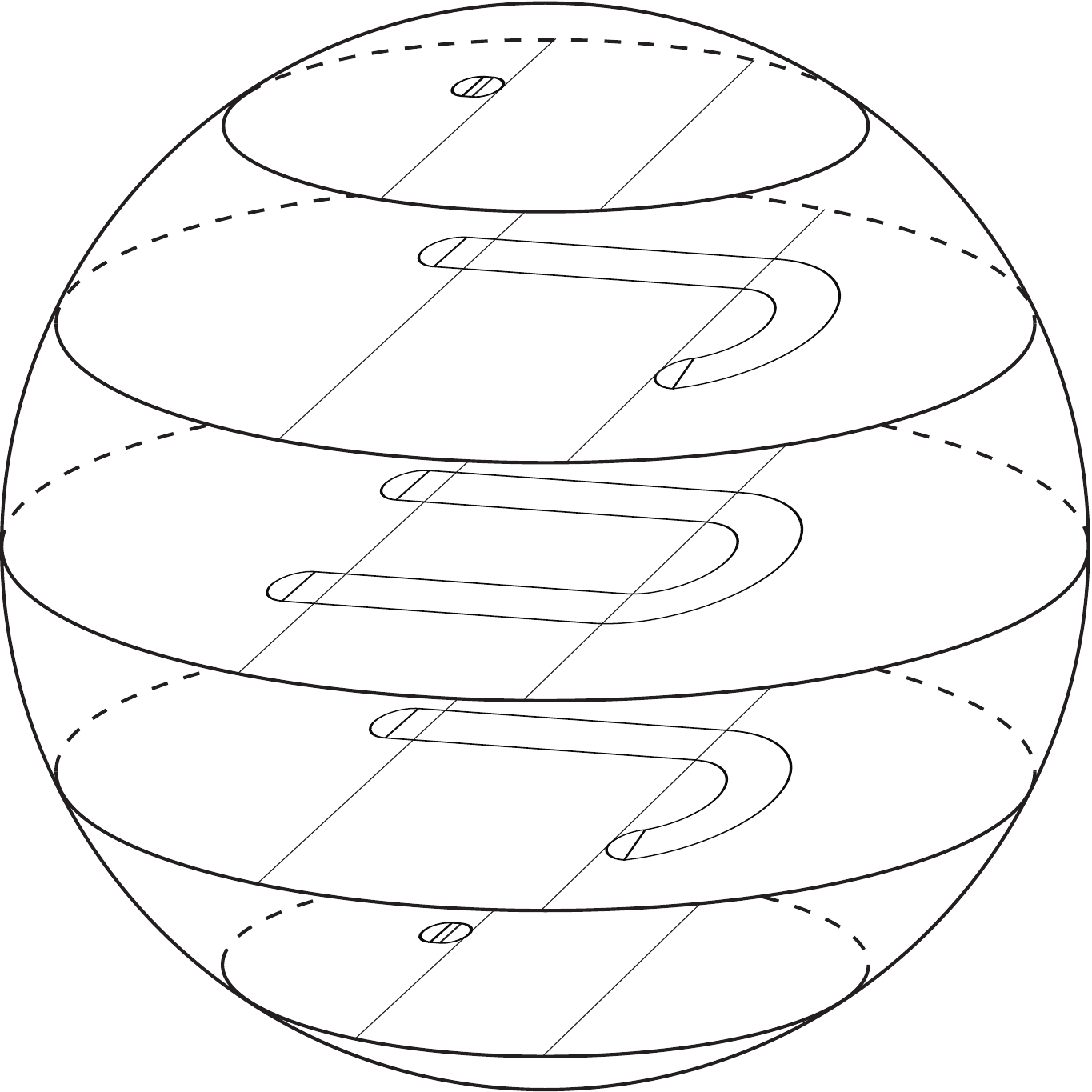}\\
  \caption{}
  \label{deformacao}
\end{minipage} \hfill
\begin{minipage}[b]{0.48\linewidth}
  \centering
  \includegraphics[width=5.5cm]{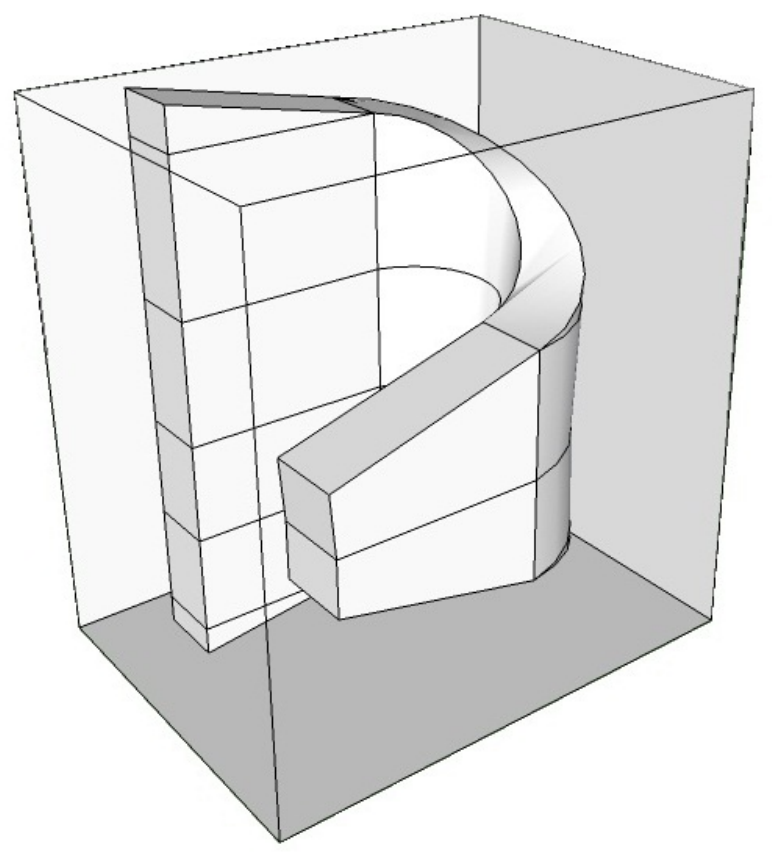}\\
  \caption{}
  \label{}
 \end{minipage}
\end{figure}

Consider now a $C^{\infty}-$flow $\phi(t,(x,y,z))$ on the $3$-dimensional disc $\mathbb{D}^3=\{(x,y,z)\in \R^3 : x^2+y^2+z^2\le1\}$ with just two singularities: the north pole $N=(0,0,1)$ and the south pole $S=(0,0,-1)$, and such that the orbit of each point in $\mathbb{D}^3\setminus \{N,S\}$ has $N$ as $\alpha$-limit and $S$ as $\omega$-limit. Moreover,  suppose that for each orbit the time $t$ is parametrized such that each $2$-dimensional disc $\mathbb{D}^2_{z_0} = \{(x,y,z)\in\mathbb{D}^3 : z=z_0 \}$, for $z_0\in(-1,1)$, is send in another disc $\mathbb{D}^2_{z_1}$, with $z_1<z_0$.
Denote by $\phi_{\tau}:\mathbb{D}^3\to\mathbb{D}^3$ the time $\tau$ diffeomorphism of this flow. It forms an one parameter family of diffeomorphisms such that in $\tau=0$ we get the identity map on $\mathbb{D}^3$.

Considere the family of diffeomorphisms $G_{\tau} = \phi_{\tau}\circ F$. For $\tau=0$, we have that $G_{0}=F$, that the disc $\mathbb{D}^2_{0}$ is $F$-invariant and $F|_{\mathbb{D}^2_{0}}=f$. Furthermore, the entropy of $f$ is $h(f)=\log 2$, which implies that $h(F) \ge h(F|_{\mathbb{D}^2_{0}})=\log 2$. Nevertheless, if $\tau>0$ is small, the entropy of $G_{\tau}$ vanishes, because the set $\Omega(G_{\tau})$ is constituted just by $S$ and $N$.

In short, we got an arc of diffeomorphisms $G_{\tau}$ such that for each $\tau>0$ the entropy of $G_{\tau}$ is zero and for $\tau=0$ the entropy jumps to $\log 2$. It proves that the topological entropy is not lower semicontinuous in $G_{0}$.
\end{proof}

%==========================================================================================
%%%%%%%%%%%%%%%%%%%%%%%%%%%%%%%%%%%%%%%%%%%%%%%%%%%%%%%%%%%%%%%%%%%%%%%%%%%%%%%%%%%%%%%%%%%
%%%%%%%%%%%%%%%%%%%%%%%%%%%%%%%%%%%%%%%  Referï¿½cias Bibliogrï¿½icas   %%%%%%%%%%%%%%%%%%%%%
%%%%%%%%%%%%%%%%%%%%%%%%%%%%%%%%%%%%%%%%%%%%%%%%%%%%%%%%%%%%%%%%%%%%%%%%%%%%%%%%%%%%%%%%%%%

%=====================================================================================================================

\end{document}